\documentclass[12pt,twoside]{amsart}
\usepackage{amssymb,amscd}
\usepackage{mathrsfs}
\usepackage{array,float}
\usepackage[all]{xy}
\usepackage{enumerate}

\theoremstyle{plain}
\newtheorem{thm}{Theorem}[section]
\newtheorem{thmx}{Theorem}

\newtheorem{corx}[thmx]{Corollary}

\newtheorem{lem}[thm]{Lemma}
\newtheorem{pro}[thm]{Proposition}

\newtheorem{question}[thm]{Question}

\theoremstyle{remark}
\newtheorem{rem}[thm]{Remark}

\theoremstyle{definition} 
\newtheorem{definition}[thm]{Definition}

\numberwithin{equation}{section}

\newcommand{\Z}{\mathbb{Z}}

\renewcommand{\phi}{\varphi}

\begin{document}

\author{Ilir Snopce}
\address{Universidade Federal do Rio de Janeiro\\
  Instituto de Matem\'atica \\
  21941-909 Rio de Janeiro, RJ \\ Brasil }
\email{ilir@im.ufrj.br}

\thanks{This research was partially supported by CNPq}

\title[Asymptotic density of test elements] {Asymptotic density of test elements in free groups and surface groups}

\author{Slobodan Tanushevski} \address{Universidade Federal do Rio de Janeiro\\
  Instituto de Matem\'atica \\
  21941-909 Rio de Janeiro, RJ \\ Brasil }

\email{tanusevski1@gmail.com}

\begin{abstract}
An element $g$ of a group $G$ is a test element  if every endomorphism of $G$ that fixes $g$ is an automorphism.
Let $G$ be a free group of finite rank, an orientable surface group of genus $n \geq 2$, or a non-orientable surface group of genus $n \geq 3$.
Let $\mathcal{T}$ be the set of test elements of $G$.
We prove that $\mathcal{T}$  is a net. From this result we derive that $\mathcal{T}$ has positive asymptotic density in $G$.
This answers a question of Kapovich, Rivin, Schupp, and Shpilrain.  
Furthermore, we prove that $\mathcal{T}$ is dense in the profinite topology on $G$.
\end{abstract}


\maketitle

\section{Introduction}

Let $G$ be a finitely generated group with a finite generating set $X$. Given $g \in G$, we denote by $|g|_X$ the smallest integer $k \geq 0$ for which there exist $x_1, \ldots, x_k \in X^{\pm 1}$ such that $g=x_1 \ldots x_k$.
The \emph{word metric} on $G$ with respect to $X$ is defined by $d_X(g,h)=|g^{-1}h|_X$ for $g,h \in G$.
A subset $S$ of $G$ is called a $C$-\emph{net} ($0 \leq C < \infty$) with respect to $X$ if 
\[d_X(g,S)=\inf\{d_X(g, s) \mid s \in S\} \leq C\]
for all $g \in G$. It is easy to show that the property of being a net (a $C$-net for some constant C) does not depend on the choice of a finite generating set $X$.


Let $G$ be a group. An element $g \in G$ is a \emph{test element} of $G$ if every endomorphism of $G$ that fixes $g$ is an automorphism. The notion of a test element was introduced by Shpilrain \cite{Shpilrain} and has since become a subject of active research in group theory. Test elements have been also studied in other  algebraic structures such as polynomial algebras and Lie algebras (see \cite{Mikhalev} and the references therein).

Let $F(x_1, \ldots, x_n)$ be a free group with basis $\{x_1, \ldots, x_n \}$. In 1918, Nielsen \cite{Nielsen} proved that the commutator $[x_1, x_2]$ is a test element of $F(x_1, x_2)$.  Zieschang \cite{Zieschang1} generalized this result by proving  that $[x_1, x_2][x_3,x_4]\ldots [x_{2m-1}, x_{2m}]$ is a test element of $F( x_1, x_2, ..., x_{2m})$. 
Further, Rips \cite{Rips} showed that every higher commutator $[x_1, x_2, ... , x_n]$ is a test element of $F(x_1, ..., x_n)$. 
In another direction, Turner  \cite{Turner} proved that  $x_1^{k_1}x_2^{k_2}\ldots x_n^{k_n}$  is a test element of $F(x_1, ..., x_n)$ if and only if $k_i \neq 0$ for all $1 \leq i \leq n$ and $\gcd(k_1, \ldots, k_n) \neq 1$ (see also \cite{Zieschang2}).  

Let $G = \langle  x_1, \ldots, x_{2n}     |    [x_1, x_2]\ldots [x_{2n-1}, x_{2n}]  \rangle$  be an orientable surface group of genus $n\geq 2$.
O'Neill and Turner \cite{O'Neill} proved that $x_1^kx_2^k \ldots x_{2n}^k$  is a test element of $G$ for every $k \geq 2$. In \cite{Slobodan}, we study test elements in pro-$p$ groups and use our results to give multitude of new examples of test elements in free and surface groups. In this short note we utilize the pro-$p$ techniques developed in \cite{Slobodan} to describe the distribution of test elements in these two classes of groups.   

\begin{thmx}
\label{free}
The set of test elements of a free group $F$ of rank $n\geq 2$  is a $(3n-2)$-net  with respect to every finite generating set of $F$.
\end{thmx}

\begin{thmx}
\label{surface}
\begin{enumerate}[(i)]
\item The set of test elements of an orientable surface group $G$ of genus $n\geq 2$ is a $(161n+8 \cdot 25^n(n-1)(16n+1)+33)$-net  with respect to every finite generating set of $G$.
\item The set of test elements of a non-orientable surface group $G$ of genus $n\geq 3$ is a $(5n-5)$-net  with respect to every finite generating set of $G$.
\end{enumerate}
\end{thmx}
Let $G$ be a finitely generated group with a finite generating set $X$.  We denote by $B_X(r)$
the ball of radius $r \geq 0$ centered at the identity in the metric space $(G, d_X)$. 
Given $S \subseteq G$, the \emph{asymptotic  density} of $S$ in $G$ with respect to $X$ is defined as
\[\overline{\rho}_X(S)= \limsup_{k \to \infty} \frac{|S \cap B_{X}(k)|}{|B_X(k)|}. \]
If the actual limit exists,  we refer to it as the \emph{strict asymptotic density} of $S$ in $G$ with respect to $X$, and we write $\rho_X(S)$ instead of $\displaystyle{\overline{\rho}_X(S)}$.
A subset $S$ of $G$ is \emph{generic} in $G$ (with respect to $X$) if $\rho_X(S)=1$, and 
it is  \emph{negligible} if $\rho_X(S)=0$.

Most of the subsets of free groups for which the asymptotic density has been studied and which could be defined by a natural algebraic property are either generic or negligible
(cf.  \cite{Woess}, \cite{Olshanskii},  \cite{Kapovich2}, \cite{Burillo}, and \cite{Goldstein}). In \cite{Kapovich}, it was proved that the set of test elements of a free group of rank two has intermediate density (different from $0$ and $1$), and the following question was raised.

\begin{question}[{\cite[Problem 5.1]{Kapovich}}]
\label{question}
Let $F$ be a free group of rank $n \geq 3$. Is the set of test elements of $F$ negligible?
\end{question}

A negative answer to Question~\ref{question} will follow from Theorem~\ref{free} and the following Lemma, which is a slight extension of \cite[Proposition~2.1]{Burillo}. 

\begin{lem}
\label{basic}
Let $G$ be a finitely generated group, $X$ a finite generating set of $G$, and $S \subseteq G$.
Suppose that $G=Sg_1 \cup \ldots \cup Sg_m$ for some $g_1, \ldots, g_m \in B_X(C)$. Then
\[\overline{\rho}_X(S) \geq \liminf_{k \to \infty} \frac{|S \cap B_{X}(k)|}{|B_X(k)|} \geq \frac{1}{m|B_X(C)|}.\]
\end{lem}
\begin{proof}
Note that for every $1\leq i \leq m$ and $k \geq C$ , we have an injection from $Sg_i \cap B_X(k-C)$ into $S \cap B_X(k)$ given by $x \to xg_i^{-1}$.
Hence, $|Sg_i \cap B_X(k-C)| \leq |S \cap B_X(k)|$ and
\[\frac{|Sg_i \cap B_X(k-C)|}{|B_X(k-C)|} \leq \frac{|S \cap B_X(k)|}{|B_X(k-C)|} \leq |B_X(C)|\frac{|S \cap B_X(k)|}{|B_X(k)|}.\] 
The last inequality follows from the inequality $|B_X(k)| \leq |B_X(k-C)||B_X(C)|$. 

Clearly, $|B_X(k-C)| \leq |Sg_1 \cap B_X(k-C)| + \ldots + |Sg_m \cap B_X(k-C)|$ for every $k \geq C$. Thus
\[1\leq \frac{|Sg_1 \cap B_X(k-C)|}{|B_X(k-C)|} + \ldots + \frac{|Sg_m \cap B_X(k-C)|}{|B_X(k-C)|} \leq m|B_X(C)| \frac{|S \cap B_X(k)|}{|B_X(k)|}.\]
Now divide by $m|B_X(C)|$ and take $\liminf$. 
\end{proof}

Observe that $S \subseteq G$ is a $C$-net (with respect to $X$) if and only if there exist elements $g_1, \ldots, g_m \in B_X(C)$ such that $G=Sg_1 \cup \ldots \cup Sg_m$. 
Now the following Corollary follows form Lemma~\ref{basic} and the proofs of Theorem~\ref{free} and Theorem~\ref{surface}.
\begin{corx}
\label{density - test elements}
Let $\mathcal{T}$ be the set of test elements of a group $G$. 
\begin{enumerate}[(i)]
\item If $G$ is a free group of rank $n\geq 2$, then 
\[\overline{\rho}_X(\mathcal{T}) \geq \liminf_{k \to \infty}\frac{|\mathcal{T} \cap B_{X}(k)|}{|B_{X}(k)|} \geq  \frac{1}{(2^{n+1}(2^n-1)+1)|B_X(3n-2)|}\]
for every generating set $X$ of $G$.
\item If $G$ is a non-orientable surface group of genus $n \geq 3$, then
\[\overline{\rho}_X(\mathcal{T}) \geq \liminf_{k \to \infty}\frac{|\mathcal{T} \cap B_{X}(k)|}{|B_{X}(k)|} \geq  \frac{1}{6^{n-1}|B_X(5n-5)|}\]
for every generating set $X$ of $G$.
\item If $G$ is an orientable surface group of genus $n \geq 2$, then
\[\overline{\rho}_X(\mathcal{T}) \geq \liminf_{k \to \infty}\frac{|\mathcal{T} \cap B_{X}(k)|}{|B_{X}(k)|} \geq  \frac{1}{|B_X(161n+8 \cdot 25^n(n-1)(16n+1)+33)|^2}\]
for every generating set $X$ of $G$.
\end{enumerate}
\end{corx}

A \emph{retract} of a group $G$ is a subgroup $R \leq G$ for which there exists a homomorphism $r:G \to R$, called a \emph{retraction}, such that $r(g)=g$ for all $g \in R$.
We say that a group $G$ is a \emph{Turner group} if it satisfies the \emph{Retract Theorem}: an element $g\in G$ is a test element of $G$ if and only if it is not contained in any proper retract of $G$. 
Free groups of finite rank were the first examples of  Turner groups  (see \cite{Turner}).
Other examples are finitely generated Fuchsian groups \cite{O'Neill} and torsion free hyperbolic groups \cite{Groves}. 
The fundamental group of the Klein bottle $\langle a, b \mid aba^{-1}=b^{-1} \rangle $ is an example of a group which is not  a Turner group (see \cite{Turner}). However, all other surface groups  are Turner groups.

\begin{corx}
\label{density - retracts}
Let $\mathcal{R}$ be the union of all proper retracts of a group $G$. 
\begin{enumerate}[(i)]
\item If $G$ is a free group of rank $n\geq 2$, then 
\[\overline{\rho}_X(\mathcal{R}) \leq 1 - \frac{1}{(2^{n+1}(2^n-1)+1)|B_X(3n-2)|}\]
for every generating set $X$ of $G$.
\item If $G$ is a non-orientable surface group of genus $n \geq 3$, then
\[\overline{\rho}_X(\mathcal{R}) \leq  1 - \frac{1}{6^{n-1}|B_X(5n-5)|}\]
for every generating set $X$ of $G$.
\item If $G$ is an orientable surface group of genus $n \geq 2$, then
\[\overline{\rho}_X(\mathcal{R}) \leq  1 - \frac{1}{|B_X(161n+8 \cdot 25^n(n-1)(16n+1)+33)|^2}\]
for every generating set $X$ of $G$.
\end{enumerate}
\end{corx}

Let $F$ be a free group of rank $n \geq 2$ with basis $X$, and let $\mathcal{T}$ and $\mathcal{R}$ be, respectively, the set of test elements and the union of all proper retracts of $F$.
It follows from \cite[Theorem~B]{Kapovich} that  
\[ \overline{\rho}_X(\mathcal{T}) \leq 1- \frac{4n-4}{(2n-1)^2 \zeta (n)},\]
where $\zeta(z)=\sum_{n=1}^{\infty} n^{-z}$ is the Riemann zeta-function. Consequently, 
\[ \overline{\rho}_X(\mathcal{R}) \geq \liminf_{k \to \infty}\frac{|\mathcal{R} \cap B_{X}(k)|}{|B_{X}(k)|} \geq \frac{4n-4}{(2n-1)^2 \zeta (n)}.\]
It follows that $\mathcal{T}$ and $\mathcal{R}$ have intermediate densities in $F$.
In comparison, it is interesting to note that the union of all proper free factors of $F$ is negligible (this follows from \cite{Stallings2} and \cite{Burillo}). 

Recall that the \emph{exponential growth} of a finitely generated group $G$ with a finite generating set $X$ is defined as $\displaystyle{w_X(G)=\lim_{k \to \infty}\sqrt[k]{|B_X(k)|}}$. More generally, the 
\emph{exponential growth} of  a subset $S$ of $G$ is defined as  $\displaystyle{w_X(S)=\limsup_{k \to \infty}\sqrt[k]{|S \cap B_X(k)|}}$ (here we are forced to take $\limsup$ since the limit may not exist).
It is easy to see that if $\displaystyle{\liminf_{k \to \infty}\frac{|\mathcal{S} \cap B_{X}(k)|}{|B_{X}(k)|} > 0}$, then $\displaystyle{w_X(S)=\lim_{k \to \infty}\sqrt[k]{|S \cap B_X(k)|}=w_X(G)}$. 
In particular, $w_X(\mathcal{T})=w_X(G)$ for $G$ and $\mathcal{T}$ as in Corollary~\ref{density - test elements}. Furthermore, $w_X(\mathcal{R})=w_X(G)$ when $G$ is a 
free group of rank $n \geq 2$ and $\mathcal{R}$ is the union of all proper retracts of $G$.

\medskip
Recall that the \emph{profinite topology} on a group $G$ is the topology with basis 
\[\{gN \mid g \in G, N \unlhd G \text{ and } |G:N| < \infty\}.\] 
By \cite[Theorem~7.5]{Slobodan}, the set of test elements of a free group of finite rank is dense in the profinite topology.  
Here we prove the same for surface groups.

\begin{thmx}
\label{surface - dense}
Let $G$ be an orientable surface group of genus $n \geq 2$ or a non-orientable surface group of genus $n \geq 3$. Then
the set of test elements of $G$  is dense in the profinite topology.
\end{thmx}

\section{Preliminaries}
In this section, for the convenience of the reader, we collect the results from \cite{Slobodan} that will be used in the paper.  
We start by fixing some notation.
Throughout the paper, $p$ denotes a prime. The $p$-adic integers are denoted by $\Z_p$.
Subgroup $H$ of a pro-$p$ group $G$ is tacitly taken to be closed and by generators we mean topological generators as
appropriate. If $G$ is a pro-$p$ group and $X \subseteq G$, then $\langle X \rangle$ is the closed subgroup of $G$ topologically generated by $X$.  
The minimal number of generators of a pro-$p$ group $G$ is denoted by $d(G)$. Homomorphisms between pro-$p$ groups are always assumed to be continuous.
We write $\Phi^{n}(G)$ for the n-th term of the Frattini series of a pro-$p$ group $G$. 

Let $G$ be a discrete group. We denote by $\jmath_{p}:G \to \widehat{G}_{p}$ the pro-$p$ completion of $G$. 
If $G$ is residually finite-$p$, then $\jmath_p$ is an embedding and we identify $\jmath_{p}(G)$ with $G$. 
As before, $F(x_1, \ldots, x_n)$ stands for a free discrete group with basis $X=\{x_1, \ldots, x_n \}$. The pro-$p$ completion of $F$ is
denoted by $\widehat{F}_p(x_1, \ldots, x_n)$ and it is a free pro-$p$ group with basis $X$.  

\begin{thm}[ {\cite[Theorem~3.5  and  Corollary~3.6]{Slobodan}} ]
\label{Turner}
\begin{enumerate}[(a)]
\item The test elements of a finitely generated  profinite group $G$ are exactly the elements not contained in any proper retract of $G$.
\item The only retracts of a free pro-$p$ group are the free factors.
The test elements of a free pro-$p$ group $\widehat{F}_p$ of finite rank are the elements not contained in any proper free factor of $\widehat{F}_p$.
\end{enumerate}
\end{thm}

The following proposition allows us to transfer results related to test elements from the pro-$p$ to the discrete category.

\begin{pro}[{\cite[Proposition~7.1]{Slobodan}}]
\label{pro-p to discrete}
Let $p$ be a prime, and let $G$ be a finitely generated residually finite-$p$ Turner group. 
If $w \in G$ is a test element of $\widehat{G}_p$, then $w$ is a test element of $G$. 
\end{pro}

Note that orientable surface groups and non-orientable surface groups of genus $n\geq 4$ are residually free, and therefore residually finite-$p$ for every prime $p$ since free groups are residually finite-$p$ (see \cite{Baumslag1} and \cite{Benjamin-Baumslag}).  On the other hand, the non-orientable surface group of genus 3 is finite-$p$ by \cite[Lemma~8.9]{Minasyan}.
Hence, free groups, orientable surface groups and non-orientable surface groups of genus $n\geq 3$ satisfy the hypothesis of Proposition~\ref{pro-p to discrete}.  

\begin{pro}[{\cite[Proposition~5.13 c)]{Slobodan}}]
\label{free - test element}
The element $x_1^p \ldots x_n^p$ is a test element of $\widehat{F}_p(x_1, \ldots, x_n)$.
\end{pro}

\begin{pro}[{\cite[Proposition~5.7]{Slobodan}}]
\label{powers}
If $u$ is a test element of $\widehat{F}_p(x_1, \ldots, x_n)$, then for any $0 \neq \alpha_i \in \mathbb{Z}_p$, $1 \leq i \leq n$, $ u(x_1^{\alpha_1}, \ldots, x_n^{\alpha_n})$ is also a test element of $\widehat{F}_p$.
\end{pro}

Let $G$ be a pro-$p$ group. A non-primitive element $g \in G$ is called \emph{almost primitive} element of $G$ 
if it is a primitive element of every  proper subgroup of $G$ that contains it.
Thus $g \in G$ is an almost primitive element of $G$ if $g \in \Phi(G)$ and $g \notin \Phi(H)$ whenever $H \lneq  G$.
It is easy to see that an almost primitive element of a finitely generated pro-$p$ group is a test element (see \cite[Proposition~5.10]{Slobodan}).


\begin{lem}[{\cite[Lemma~5.17]{Slobodan}}]
\label{extension - free}
Let $G$ be a pro-$p$ group and $w \in \Phi(G)$. Then the following holds:
\begin{enumerate}[(a)]
\item there exists $H \leq G$  such that $w \in H$ and $w$ is an almost primitive element of $H$;
\item if $G$ is a free pro-$p$ group with basis $\{a_1, \ldots, a_n\}$, then there exists a subset $\{a_{i_1}, \ldots, a_{i_m}\} \subseteq \{a_1, \ldots, a_n\}$ such that
for every test element $u$ of $\widehat{F}_p(x_1, \ldots, x_m)$ contained in $\Phi(\widehat{F}_p)$, $w \cdot u(a_{i_1}, \ldots, a_{i_m})$ is a test element of $G$. Moreover, if $w \neq 1$, then
we can choose $m < n$.
\end{enumerate}
\end{lem}

The following Lemma follows easily from the proof of Lemma~\ref{extension - free}.

\begin{lem}
\label{extension - free2}
Let $G$ be a free pro-$p$ group of finite rank and $H \leq G$. If $G= \langle H \cup \{a_1, \ldots, a_n \} \rangle$ and $w \in H$ is an almost primitive element of $H$, then there exists a subset $\{a_{i_1}, \ldots, a_{i_m}\} \subseteq \{a_1, \ldots, a_n\}$ such that for every test element $u$ of $\widehat{F}_p(x_1, \ldots, x_m)$ contained in $\Phi(\widehat{F}_p)$, $w \cdot u(a_{i_1}, \ldots, a_{i_m})$ is a test element of $G$.
\end{lem}

We will also need the following special case of  {\cite[Proposition~2.2]{Slobodan}}.

\begin{lem}
\label{free factor}
Let $G$ be a free pro-$p$ group, $A$ a free factor of $G$, and $H$ an arbitrary subgroup of $G$. 
Then $A \cap H$ is a free factor of $H$.
\end{lem}

\section{Proofs of the theorems}
\begin{proof}[Proof of Theorem~\ref{free}]
Let $X$ be any generating set of $F$. Then $X$ is also a generating set of the free pro-$2$ group $\widehat{F}_2$.
Note that every generating set of a free pro-$p$ group contains a basis. Let $X'=\{x_1, \ldots, x_n\} \subseteq X$ be a basis of $\widehat{F}_2$. 
By Proposition~\ref{free - test element}, $s=x_1^2 \ldots x_n^2$ is a test element of $\widehat{F}_2$, and it follows from Proposition~\ref{pro-p to discrete} 
that $s$ is also a test element of $F$. Observe that $d_X(1, s)=|s|_X=2n \leq 3n-2$. 

Now let $1 \neq w \in F$.
For $1 \leq i \leq n$, define $\sigma_{x_i}: \widehat{F}_2 \to \mathbb{Z}_2$ by $\sigma_{x_i}(x_j)=\delta_{ij}$ ($\delta_{ij}$ is the Kronecker delta). 
It is easy to see that there exist $\epsilon_i \in \{0, 1\}$, $1 \leq i \leq n,$ such that $\epsilon_i + \sigma_{x_i}(w) \in 2\mathbb{Z}_2$. 
Set $u=w$ if $\epsilon_i=0$ for all $1 \leq i \leq n$. If $\epsilon_j=1$ for some $1 \leq j \leq n$, then
$p_1=w x_1^{\epsilon_1} \ldots x_i^{\epsilon_j} \ldots x_n^{\epsilon_n} \neq 1$ or $p_2=w x_1^{\epsilon_1} \ldots x_j^{-\epsilon_j} \ldots x_n^{\epsilon_n} \neq 1$; let $u=p_1$ if $p_1 \neq 1$, else set $u=p_2$.
In any case, it is easy to see that $u \in \Phi(\widehat{F}_2)$. 

It follows from Lemma~\ref{extension - free} (b) and Proposition~\ref{free - test element} that
there exists a subset $\{x_{i_1}, \ldots, x_{i_m}\} \subsetneq X'$ such that
$t=ux_{i_1}^2 \ldots x_{i_m}^2$ is a test element of $\widehat{F}_2$. By Proposition~\ref{pro-p to discrete}, $t$ is also a test element of $F$. Clearly, $d_X(w, t)  \leq n+ 2(n-1)= 3n - 2$.
Hence, $\mathcal{T}$ is a  $(3n - 2)$-net with respect to $X$.
\end{proof}

\begin{definition} Let $n\geq 1$ be an integer. A pro-$p$ group $G$ is called a \emph{Poincar\'{e} group of dimension $n$}  if it satisfies the following conditions:
\begin{itemize}
\item[(i)] $\textrm{dim} _{\mathbb{F}_p} H^i(G,\mathbb{F}_p) < \infty$  for all $i$,
\item[(ii)] $\textrm{dim} _{\mathbb{F}_p} H^n(G,\mathbb{F}_p)=1$, \textrm{ and }
\item[(iii)] the cup-product $H^i(G,\mathbb{F}_p)\times H^{n-i}(G,\mathbb{F}_p) \to H^n(G,\mathbb{F}_p) \cong \mathbb{F}_p$ is a non-degenerate bilinear form for all $0 \leq i \leq n$.
\end{itemize}
\end{definition}
A Poincar\'{e}  group of dimension 2 is called a \emph{Demushkin group} (see \cite{Serre1}). Examples of Demushkin groups are the pro-$p$ completions of orientable surface groups.
 Demushkin groups also arise in algebraic number theory. For instance, if $k$ is a $p$-adic number field (a finite extension of $\mathbb{Q}_p$) containing a primitive $p$-th root of unity and $k(p)$ is the maximal $p$-extension of $k$, then $\textrm{Gal}(k(p)/k)$ is a Demushkin group (see \cite{Wingberg}). 
Demushkin groups have been classified completely by  Demushkin, Serre and Labute (see \cite{Demushkin1}, \cite{Demushkin2}, \cite{Serre2}, and  \cite{Labute}). 

In this paragraph we  summarise some of the features of Demushkin groups that will be used in this paper.  Every finite index subgroup $H$ of a Demushkin group $G$ is a Demushkin  group of rank $d(H) = 2 + [G:H](d(G)-2)$. Moreover, every subgroup of infinite index of a Demushkin group is free pro-$p$. The rank of a  Demushkin pro-$p$ group is even whenever $p$ is an odd prime. 

The following Lemma, proved in \cite{Slobodan}, is an immediate consequence of a result of Sonn \cite{Sonn}.

\begin{lem}\label{retract - Demushkin}
Let $G$ be a Demushkin group with $d(G)=n$ and $R$ be a proper retract of $G$. Then
$d(R) \leq \frac{n}{2}$.
\end{lem}

\begin{lem}
\label{retract fratini intersection}
Let $G$ be a pro-$p$ group and $R$ be a retract of $G$. Then $\Phi^n(G) \cap R=\Phi^n(R)$ for every $n \geq 1$.
\end{lem}
\begin{proof}
For every $n \geq 1$, $\Phi^n(R) \subseteq \Phi^n(G)$. Let $r:G \to R$ be a retraction. Then
\[\Phi^n(G) \cap R=r(\Phi^n(G) \cap R)\subseteq r(\Phi^n(G)) \cap r(R)=\Phi^n(R) \cap R=\Phi^n(R).\]
Hence, $\Phi^n(G) \cap R=\Phi^n(R)$ for every $n \geq 1$.
\end{proof}

\begin{lem}
\label{extension - Demushkin}
Let $G$ be a Demushkin pro-$p$ group with $p \geq 5$ and $d(G) > 2$. Let $\{a_1, \ldots, a_{2n}\}$ be a minimal generating set of $G$.
If $w \in \Phi^{2}(G)$, then there exists a subset $\{a_{i_1}, \ldots, a_{i_m}\} \subseteq \{a_1, \ldots, a_{n+1}\}$ such that for every test element $u$ of $\widehat{F}_p(x_1, \ldots, x_m)$, $w \cdot u(a_{i_1}^p, \ldots, a_{i_m}^p)$ is a test element of $G$. 
\end{lem}
\begin{proof}
Let $w \in \Phi^2(G)$. If $w$ is a test element of $G$, we are done, otherwise it follows from Theorem~\ref{Turner} $(a)$ that $w$ is contained in a proper retract $R$ of $G$.
By Lemma~\ref{retract fratini intersection}, $w \in \Phi^{2}(R)=R \cap \Phi^2(G)$, and it follows from Lemma~\ref{extension - free} (a) that there exists $H \lneq R$  such that $w \in H$ and $w$ is an almost primitive element of $H$.
 
Set $K=\langle H \cup\{a_1^p, \ldots, a_{n+1}^p\}\rangle$. We claim that $K$ is a free pro-$p$ group.  Let $L=\langle R \cup\{a_1^p, \ldots, a_{n+1}^p\}\rangle$.
Since $K \leq L$, it suffices to show that $L$ has infinite index in $G$.  
As $\{a_1^p, \ldots, a_{n+1}^p\} \subseteq \Phi(G)$ and $d(R) \leq n$ by Lemma~\ref{retract - Demushkin}, we have $L \neq G$. Let $U$ be a proper open subgroup of $G$. Then
\[d(U) \geq p(2n-2)+2 >  2n + 1 \geq d(L),\]
and thus $L \neq U$. 

By Lemma~\ref{extension - free2}, there exists a subset $\{a_{i_1}^{p}, \ldots, a_{i_m}^{p}\} \subseteq \{a_1^{p}, \ldots, a_{n+1}^{p}\}$ such that for every test element $u$ of $\widehat{F}_p(x_1, \ldots, x_m)$ contained in $\Phi(\widehat{F}_p)$, 
$t=w \cdot u(a_{i_1}^{p}, \ldots, a_{i_m}^{p})$ is a test element of $K$. We want to prove that $t$ is also a test element of $G$.

Suppose to the contrary that $t$ is not a test element of $G$. Then it follows from Theorem~\ref{Turner} $(a)$ that $t$ is contained in a proper retract $R'$ of $G$.
Set $T= \langle K \cup R' \rangle$. We will show that $T$ is free pro-$p$. To that end, let $M$ be a maximal subgroup of $R$ that contains $H$ (recall that $H \lneq R$), and let $P=\langle M \cup \{a_1^p, \ldots, a_{n+1}^p\} \cup R' \rangle$.
Clearly, $T \leq P$, and hence, it suffices to prove that $P$ has infinite index in $G$. First note that $P \neq G$. Indeed, since $\{a_1^p, \ldots, a_{n+1}^p\} \subseteq \Phi(G)$, $d(R') \leq n$ (by Lemma~\ref{retract - Demushkin}), and 
\[\begin{split}
&\dim_{\mathbb{F}_p} (M\Phi(G) \mathbin/ \Phi(G))=\dim_{\mathbb{F}_p} (M\mathbin/ (\Phi(G) \cap M))\\
=&\dim_{\mathbb{F}_p} (M\mathbin/ (\Phi(R) \cap M)) =\dim_{\mathbb{F}_p}(M \mathbin/ \Phi(R)) \leq n-1,
\end{split}\]
it follows that $\dim_{\mathbb{F}_p}(P\Phi(G) \mathbin/ \Phi(G)) \leq 2n- 1 < 2n = \dim_{\mathbb{F}_p}(G \mathbin/ \Phi(G))$.

By the Schreier index formula, $d(M)=p(d(R) - 1)+1 \leq p(n-1)+1$.
Hence, $d(P) \leq p(n - 1) + 2n+ 2$. Let $U$ be a proper open subgroup of $G$. Then 
\[d(U) \geq p(2n-2)+2 >  p(n - 1) + 2n+ 2 \geq d(P).\]
Hence, $P$ has infinite index in $G$, and consequently $T$ is a free pro-$p$ group. 

Note that $R'$ is a retract of $T$. As $T$ is free pro-$p$, it follows from Theorem~\ref{Turner} $(b)$ that $R'$ is a free factor of $T$.
By Lemma~\ref{free factor}, $K \cap R'$ is a free factor of $K$ containing $t$. Moreover, we will show that $K \cap R'$ is a proper free factor of $K$, which will contradict the fact that $t$ is a test element of $K$.
Suppose to the contrary that $K \leq R'$. Then $\{a_1^p, \ldots, a_{n+1}^p\} \subseteq R'$. 
Let $r:G \to R'$ be a retraction. Then, for all $1 \leq i \leq n+1$, $a_i^p=r(a_i^p)=r(a_i)^p$, and by unique extraction of roots in the free pro-$p$ group $\langle a_i, r(a_i) \rangle$, 
$a_i=r(a_i) \in R'$. Hence, $n \geq d(R') \geq n+1$, a contradiction. 
\end{proof}

\begin{rem}
Note that the proof of Lemma~\ref{extension - Demushkin} remains valid if we assume $d(G) > 4$ and $p=3$. This observation can be used to increase the lower bound in Corollary~\ref{density - test elements} $(iii)$.
\end{rem}

\begin{proof}[Proof of Theorem~\ref{surface}]
$(i)$ Let $X$ be a generating set of $G$, and 
let $w \in G$. Then $X$ is also a generating set of the pro-$5$ completion $\widehat{G}_5$ of $G$, which, as noted earlier, is a Demushkin group.
Let $X'=\{x_1, \ldots, x_{2n}\} \subseteq X$ be a minimal generating set of $\widehat{G}_5$. For each $1 \leq i \leq 2n$, there is a unique homomorphism $\sigma_{x_i}:\widehat{G}_5 \to \mathbb{Z}_5$ satisfying $\sigma_{x_i}(x_j)=\delta_{ij}$.
Choose $\alpha_i \in \{0,1,2,3,4\}$, $1 \leq i \leq 2n$, such that $\alpha_i + \sigma_{x_i}(w) \in 5\mathbb{Z}_5$. 
Set $u=wx_1^{\alpha_1}\ldots x_{2n}^{\alpha_{2n}}$. It is easy to see that $u \in \Phi(\widehat{G}_5)$.

There is an obvious choice of a right transversal of $\Phi(\widehat{G}_5)$ in $\widehat{G}_5$, namely,
\[T=\{x_1^{s_1} \ldots x_{2n}^{s_n} \mid  0 \leq s_i \leq 4 \text{ for all } 1\leq i \leq 2n\}.\]
It follows from  Schreier's Lemma  that $\Phi(\widehat{G}_5)$ has a minimal generating set $Y=\{y_1, \ldots, y_k\}$, $k=d(\Phi(\widehat{G}_5))=2+5^{2n}(2n-2)$, satisfying
$\max\{|y_i|_X \mid 1\leq i \leq k\} \leq 16n + 1$. 
Since $\Phi(\widehat{G}_5)$ is a pro-$5$ completion of a surface group, 
similarly as before, for each $1 \leq i \leq k$, we have a unique homomorphism $\sigma_{y_i}:\Phi(\widehat{G}_5) \to \mathbb{Z}_5$  satisfying $\sigma_{y_i}(y_j)=\delta_{ij}$. Let
$\beta_i \in \{0,1,2,3,4\}$ be such that $\beta_i + \sigma_{y_i}(u) \in 5\mathbb{Z}_5$.
Set $v=uy_1^{\beta_1}\ldots y_{k}^{\beta_{k}}$. Then $v \in \Phi^2(\widehat{G}_5)$. By Lemma~\ref{extension - Demushkin} and Proposition~\ref{free - test element}, there is a subset $\{x_{i_1}, \ldots, x_{i_m}\} \subseteq \{x_1, \ldots, x_{n+1}\}$
such that $t=vx_{i_1}^{5^2} \ldots x_{i_m}^{5^2}$ is a test element of $\widehat{G}_5$. By Proposition~\ref{pro-p to discrete}, $t$ is also a test element of $G$.
Observe that $d_X(w,t) \leq 8n+4(2+5^{2n}(2n-2))(16n + 1)+5^2(n+1)=161n+8 \cdot 25^n(n-1)(16n+1)+33$.

\smallskip

$(ii)$ Recall that $G=\langle x_1, \ldots, x_n \mid x_1^2 \ldots x_n^2 \rangle$. Moreover, the pro-$3$ completion $\widehat{G}_3$ of $G$ has the same presentation as a pro-$3$ group (see \cite[Lemma~2.1]{Lu}). 
Observe that $t= x_1^2 \ldots x_n^2$ is a primitive element in the free pro-$3$ group $\widehat{F}_3(x_1, \ldots, x_n)$.
Hence, $\widehat{G}_3$ is a free pro-$3$ group of rank $n-1$. Now the proof of $(ii)$ is similar to the proof of Theorem~\ref{free}.

\end{proof}

\begin{proof}[Proof of Theorem~\ref{surface - dense}]
We only give the proof for orientable surface groups; the proof for non-orientable surface groups is similar.
Let $G$ be an orientable surface group of genus $n \geq 2$, and let $\{x_1, \ldots, x_{2n}\}$ be a minimal generating set of $G$.
Let $w \in G $ and $N \unlhd G$ with $|G:N| = l < \infty$.
We need to show that the coset $wN$ contains a test element of $G$.
Choose a prime $p$ such that $p \nmid l$ and $p \geq 5$. 
For each $1 \leq i \leq 2n$, we can find a natural number 
$r_i \geq 0$ satisfying $p \mid \sigma_{x_i}(w) + r_il$ where $\sigma_{x_i}:G \to \mathbb{Z}$ is defined by $\sigma_{x_i}(x_j)=\delta_{ij}$. Set $u=wx_1^{r_1l} \ldots x_{2n}^{r_{2n}l}$.
Then $wN=uN$ and $u \in \Phi(\widehat{G}_p)$. As in the proof of Theorem~\ref{surface}, let $Y=\{y_1, \ldots, y_k\}$, $k=d(\Phi(\widehat{G}_p))=2+p^{2n}(2n-2)$, be a minimal generating set of $\Phi(\widehat{G}_p)$.
Let $s_i \geq 0, 1 \leq i \leq k$, be natural numbers that satisfy $\sigma_{y_i}(u) + s_il \in p\mathbb{Z}_p$ where $\sigma_{y_i}$ is defined as in the proof of  Theorem~\ref{surface}. 
Let $v=uy_1^{s_1l} \ldots y_{k}^{s_{k}l}$. Then $v \in \Phi^2(\widehat{G}_p)$ and $vN=uN=wN$.

By Lemma~\ref{extension - Demushkin}, Proposition~\ref{free - test element}, and Proposition~\ref{powers}, there is a subset $\{x_{i_1}, \ldots, x_{i_m}\} \subseteq \{x_1, \ldots, x_{n+1}\}$
such that $t=vx_{i_1}^{p^2l} \ldots x_{i_m}^{p^2l}$ is a test element of $\widehat{G}_p$. 
It follows from Proposition~\ref{pro-p to discrete} that $t$ is a test element of $G$. Note that $t \in vN=wN$.
\end{proof}

\bibliographystyle{plain}

\end{document}